\documentclass[11pt,a4paper]{amsart}

\usepackage{amsmath}
\usepackage{mathtools}
\usepackage{amsthm}
\usepackage{amssymb}
\usepackage{hyperref}

\newcommand*\fullref[3][\relax]{%
  \ifdefined\hyperref%
    {\hyperref[#3]{#2\penalty 200\ \ref*{#3}#1}}%
  \else%
    {#2\penalty 200\ \relax\ref{#3}#1}%
  \fi%
}

\usepackage{tikz}
\usetikzlibrary{calc}

\usetikzlibrary{intersections}

\makeatletter

\pgfkeys{/pgf/.cd,
  rectangle corner radius/.initial=3pt
}
\newif\ifpgf@rectanglewrc@donecorner@
\def\pgf@rectanglewithroundedcorners@docorner#1#2#3#4{%
  \edef\pgf@marshal{%
    \noexpand\pgfintersectionofpaths
      {%
        \noexpand\pgfpathmoveto{\noexpand\pgfpoint{\the\pgf@xa}{\the\pgf@ya}}%
        \noexpand\pgfpathlineto{\noexpand\pgfpoint{\the\pgf@x}{\the\pgf@y}}%
      }%
      {%
        \noexpand\pgfpathmoveto{\noexpand\pgfpointadd
          {\noexpand\pgfpoint{\the\pgf@xc}{\the\pgf@yc}}%
          {\noexpand\pgfpoint{#1}{#2}}}%
        \noexpand\pgfpatharc{#3}{#4}{\cornerradius}%
      }%
    }%
  \pgf@process{\pgf@marshal\pgfpointintersectionsolution{1}}%
  \pgf@process{\pgftransforminvert\pgfpointtransformed{}}%
  \pgf@rectanglewrc@donecorner@true
}
\pgfdeclareshape{rectangle with rounded corners}
{
  \inheritsavedanchors[from=rectangle] 
  \inheritanchor[from=rectangle]{north}
  \inheritanchor[from=rectangle]{north west}
  \inheritanchor[from=rectangle]{north east}
  \inheritanchor[from=rectangle]{center}
  \inheritanchor[from=rectangle]{west}
  \inheritanchor[from=rectangle]{east}
  \inheritanchor[from=rectangle]{mid}
  \inheritanchor[from=rectangle]{mid west}
  \inheritanchor[from=rectangle]{mid east}
  \inheritanchor[from=rectangle]{base}
  \inheritanchor[from=rectangle]{base west}
  \inheritanchor[from=rectangle]{base east}
  \inheritanchor[from=rectangle]{south}
  \inheritanchor[from=rectangle]{south west}
  \inheritanchor[from=rectangle]{south east}

  \savedmacro\cornerradius{%
    \edef\cornerradius{\pgfkeysvalueof{/pgf/rectangle corner radius}}%
  }

  \backgroundpath{%
    \pgfkeys{
      /pgf/tips=false,
    }
    \northeast\advance\pgf@y-\cornerradius\relax
    \pgfpathmoveto{}%
    \pgfpatharc{0}{90}{\cornerradius}%
    \northeast\pgf@ya=\pgf@y\southwest\advance\pgf@x\cornerradius\relax\pgf@y=\pgf@ya
    \pgfpathlineto{}%
    \pgfpatharc{90}{180}{\cornerradius}%
    \southwest\advance\pgf@y\cornerradius\relax
    \pgfpathlineto{}%
    \pgfpatharc{180}{270}{\cornerradius}%
    \northeast\pgf@xa=\pgf@x\advance\pgf@xa-\cornerradius\southwest\pgf@x=\pgf@xa
    \pgfpathlineto{}%
    \pgfpatharc{270}{360}{\cornerradius}%
    \northeast\advance\pgf@y-\cornerradius\relax
    \pgfpathlineto{}%
  }

  \anchor{before north east}{\northeast\advance\pgf@y-\cornerradius}
  \anchor{after north east}{\northeast\advance\pgf@x-\cornerradius}
  \anchor{before north west}{\southwest\pgf@xa=\pgf@x\advance\pgf@xa\cornerradius
    \northeast\pgf@x=\pgf@xa}
  \anchor{after north west}{\northeast\pgf@ya=\pgf@y\advance\pgf@ya-\cornerradius
    \southwest\pgf@y=\pgf@ya}
  \anchor{before south west}{\southwest\advance\pgf@y\cornerradius}
  \anchor{after south west}{\southwest\advance\pgf@x\cornerradius}
  \anchor{before south east}{\northeast\pgf@xa=\pgf@x\advance\pgf@xa-\cornerradius
    \southwest\pgf@x=\pgf@xa}
  \anchor{after south east}{\southwest\pgf@ya=\pgf@y\advance\pgf@ya\cornerradius
    \northeast\pgf@y=\pgf@ya}

  \anchorborder{%
    \pgf@xb=\pgf@x
    \pgf@yb=\pgf@y%
    \southwest%
    \pgf@xa=\pgf@x
    \pgf@ya=\pgf@y%
    \northeast%
    \advance\pgf@x by-\pgf@xa%
    \advance\pgf@y by-\pgf@ya%
    \pgf@xc=.5\pgf@x
    \pgf@yc=.5\pgf@y%
    \advance\pgf@xa by\pgf@xc
    \advance\pgf@ya by\pgf@yc%
    \edef\pgf@marshal{%
      \noexpand\pgfpointborderrectangle
      {\noexpand\pgfqpoint{\the\pgf@xb}{\the\pgf@yb}}
      {\noexpand\pgfqpoint{\the\pgf@xc}{\the\pgf@yc}}%
    }%
    \pgf@process{\pgf@marshal}%
    \advance\pgf@x by\pgf@xa%
    \advance\pgf@y by\pgf@ya%
    \pgfextract@process\borderpoint{}%
    \pgf@rectanglewrc@donecorner@false
    \southwest\pgf@xc=\pgf@x\pgf@yc=\pgf@y
    \advance\pgf@xc\cornerradius\relax\advance\pgf@yc\cornerradius\relax
    \borderpoint
    \ifdim\pgf@x<\pgf@xc\relax\ifdim\pgf@y<\pgf@yc\relax
      \pgf@rectanglewithroundedcorners@docorner{-\cornerradius}{0pt}{180}{270}%
    \fi\fi
    \ifpgf@rectanglewrc@donecorner@\else
      \southwest\pgf@yc=\pgf@y\relax\northeast\pgf@xc=\pgf@x\relax
      \advance\pgf@xc-\cornerradius\relax\advance\pgf@yc\cornerradius\relax
      \borderpoint
      \ifdim\pgf@x>\pgf@xc\relax\ifdim\pgf@y<\pgf@yc\relax
       \pgf@rectanglewithroundedcorners@docorner{0pt}{-\cornerradius}{270}{360}%
      \fi\fi
    \fi
    \ifpgf@rectanglewrc@donecorner@\else
      \northeast\pgf@xc=\pgf@x\relax\pgf@yc=\pgf@y\relax
      \advance\pgf@xc-\cornerradius\relax\advance\pgf@yc-\cornerradius\relax
      \borderpoint
      \ifdim\pgf@x>\pgf@xc\relax\ifdim\pgf@y>\pgf@yc\relax
       \pgf@rectanglewithroundedcorners@docorner{\cornerradius}{0pt}{0}{90}%
      \fi\fi
    \fi
    \ifpgf@rectanglewrc@donecorner@\else
      \northeast\pgf@yc=\pgf@y\relax\southwest\pgf@xc=\pgf@x\relax
      \advance\pgf@xc\cornerradius\relax\advance\pgf@yc-\cornerradius\relax
      \borderpoint
      \ifdim\pgf@x<\pgf@xc\relax\ifdim\pgf@y>\pgf@yc\relax
       \pgf@rectanglewithroundedcorners@docorner{0pt}{\cornerradius}{90}{180}%
      \fi\fi
    \fi
  }
}

\makeatother

\makeatletter

\@ifpackagelater{tikz}{2013/12/01}{
  
}{
  
}

\makeatother%

\theoremstyle{definition}
\newtheorem{definition}{Definition}[section]

\theoremstyle{plain}
\newtheorem{corollary}[definition]{Corollary}
\newtheorem{lemma}[definition]{Lemma}
\newtheorem{proposition}[definition]{Proposition}
\newtheorem{theorem}[definition]{Theorem}

\numberwithin{equation}{section}
\newcommand*{\defterm}[1]{\emph{#1}}

\makeatletter
\newcommand\chyph{\penalty\@M-\hskip\z@skip}
\newcommand\cendash{\penalty\@M--\hskip\z@skip}
\makeatother

\DeclarePairedDelimiter{\parens}{\lparen}{\rparen}

\DeclarePairedDelimiter{\set}{\{}{\}}
\DeclarePairedDelimiterX{\gset}[2]{\{}{\}}{\,#1:#2\,}

\makeatletter
\newcommand*{\biggg}{\bBigg@{4}}

\newcommand*{\Biggg}{\bBigg@{5}}

\newcommand*{\sizeddelimiter}[2]{\bBigg@{#1}#2}
\makeatother

\makeatletter
\newcommand*{\sizedsurd}[2][]{%
  {\@mathmeasure\z@{\nulldelimiterspace\z@}%
     {\sqrt[#1]{\vcenter to #2\big@size{}}}%
     \box\z@}%
 }
\makeatother

\newcommand*{\nset}{\mathbb{N}}
\newcommand*{\zset}{\mathbb{Z}}

\newcommand*{\emptyword}{\varepsilon}
\newcommand*{\rev}{\mathrm{rev}}

\DeclarePairedDelimiterX{\pres}[2]{\langle}{\rangle}{#1\,\delimsize\vert\,\mathopen{}#2}

\newcommand*\bset{\mathbb{B}}

\begin{document}

\title{Two applications of monoid actions to cross-sections}

\author{Tara Brough}
\address{%
Centro de Matem\'{a}tica e Aplica\c{c}\~{o}es\\
Faculdade de Ci\^{e}ncias e Tecnologia\\
Universidade Nova de Lisboa\\
2829--516 Caparica\\
Portugal
}
\email{%
t.brough@fct.unl.pt
}

\author{Alan J. Cain}
\address{%
Centro de Matem\'{a}tica e Aplica\c{c}\~{o}es\\
Faculdade de Ci\^{e}ncias e Tecnologia\\
Universidade Nova de Lisboa\\
2829--516 Caparica\\
Portugal
}
\email{%
a.cain@fct.unl.pt
}

\thanks{The first author was supported by an {\sc FCT} post-doctoral fellowship ({\sc SFRH}/{\sc BPD}/121469/2016). The
  second author was supported by an Investigador {\sc FCT} fellowship ({\sc IF}/01622/2013/{\sc CP}1161/{\sc
    CT}0001). The first and second authors were partially supported by by the Funda\c{c}\~{a}o para a Ci\^{e}ncia e a
  Tecnologia (Portuguese Foundation for Science and Technology) through the project {\sc UID}/{\sc MAT}/00297/2013
  (Centro de Matem\'{a}tica e Aplica\c{c}\~{o}es) and the projects {\scshape PTDC}/{\scshape MHC-FIL}/2583/2014 and
  {\scshape PTDC}/{\scshape MAT-PUR}/31174/2017. The authors thank Ian McQuillan for supplying references on the
  relationship between ET0L and LIG. The authors thanks the anonymous referee for their careful reading of the paper and
  many valuable suggestions.}

\author{Victor Maltcev}
\address{%
  c/o second author
}
\email{%
  victor.maltcev@googlemail.com
}

\begin{abstract}
  Using a construction that builds a monoid from a monoid action, this paper exhibits an example of a direct product of
  monoids that admits a prefix-closed regular cross-section, but one of whose factors does not admit a regular
  cross-section; this answers negatively an open question from the theory of Markov monoids. The same construction is
  then used to show that for any full trios $\mathfrak{C}$ and $\mathfrak{D}$ such that $\mathfrak{C}$ is not a subclass
  of $\mathfrak{D}$, there is a monoid with a cross-section in $\mathfrak{C}$ but no cross-section in
  $\mathfrak{D}$.
\end{abstract}

\maketitle

\section{Introduction}

Cross-sections (that is, sets of normal forms) of semigroups and groups have been an important area of investigation,
both in their own right and in connection with other topics such as automatic structures
\cite{epstein_wordproc,campbell_autsg}. Sometimes rather suprising results have emerged, such as the free inverse monoid
of rank $1$ having no regular cross-section \cite[Proof of Theorem~2.7]{cutting_regularnf}.

A group is Markov if it admits a prefix-closed regular language of unique representatives. This notion was introduced by
Gromov in his seminal paper on hyperbolic groups \cite[\S~5.2]{gromov_hyperbolic}, and explored further by Ghys \& de~la
Harpe~\cite{ghys_markov}. The study of the concept of being Markov was extended to semigroups and monoids by the second
and third authors \cite{cm_markov}, who proved that the class of Markov monoids is closed under direct product
\cite[Theorem~14.1(1)]{cm_markov} and asked whether this class is closed under direct factors \cite[Question
14.6]{cm_markov}.

One of the main results of this paper is the construction of an example of a monoid that does not have a regular
cross-section but whose direct product with the free group of rank $1$ \emph{does} have a regular cross-section, and
indeed a prefix-closed one.  This proves that neither the class of regular cross-section monoids nor the class of Markov
monoids is closed under direct factors. This may shed new light on the long-standing open questions on whether the
classes of automatic monoids \cite[Question~6.6]{campbell_autsg} and automatic groups \cite[Open
Question~4.1.2]{epstein_wordproc} are closed under taking direct factors. (However, the example constructed herein is
far from being automatic.)

The other main result is that for any full trios $\mathfrak{C}$ and $\mathfrak{D}$ such that $\mathfrak{C}$ is not a
subclass of $\mathfrak{D}$, there is a monoid with a cross-section in $\mathfrak{C}$ but no cross-section in
$\mathfrak{D}$.

The main tool is a construction developed by Maltcev \& Ru\v{s}kuc~\cite{maltcev_hopfian} that produces a semigroup from
a semigroup action or a monoid from a monoid action; the details of the construction are recalled in
\fullref{Subsection}{subsec:monoidactions}.

\section{Preliminaries}
\label{sec:preliminaries}

\subsection{Monoids from monoid actions}
\label{subsec:monoidactions}

This subsection recalls the construction of a monoid from a monoid action introduced (for semigroups) in \cite[\S~5]{maltcev_hopfian}.

Let $M$ be a monoid acting (from the right) on a set $T$. Define the monoid $M[T]$ to be the disjoint union
$M \cup T$ with multiplication in $M$ as before and defined elsewhere by
\[
tm = t \cdot m; \quad mt = t; \quad ty = y \quad\text{for all $t,y \in T$ and $m \in M$.}
\]
It is straightforward to check that this multiplication is associative.

\subsection{Languages and cross-sections}

For the definition of basic language-theoretic concepts, see \cite{hopcroft_automata}. Recall in particular that a
\defterm{full trio} (also called a \defterm{cone}) is a family of languages, containing at least one non-empty language,
closed under homomorphism, inverse homomorphism, and intersection with regular languages. Recall also that full trios
are closed under GSM mappings and inverse GSM mappings \cite[Theorems~11.1 \& 11.2]{hopcroft_automata}.  Full trios are
also closed under the operation of union with a regular language. Since this fact does not seem to be explicitly stated
in the literature, we give a proof here:

\begin{lemma}
  \label{lem:fulltrioclosureunionreg}
  Full trios are closed under the operation of union with a regular language.
\end{lemma}

\begin{proof}
  Let $\mathfrak{C}$ be a full trio, let $L \in \mathfrak{C}$ and let $R$ be a regular language. If $L$ is empty, then
  $L \cup R$ is regular and so $L \cup R \in \mathfrak{C}$ since a full trio contains all regular languages
  \cite[p.~271]{hopcroft_automata}. So suppose $L$ is non-empty. Let $A$ be the alphabet of $L \cup R$ and let $x$ be a
  new symbol not in $A$. Define a homomorphism that fixes each symbol in $A$ and maps $x$ to some word in $L$. Then
  $L \cup \set{x}$ is also in $\mathfrak{C}$ by the closure of full trios under inverse homomorphism. Now consider the
  substitution that fixes each symbol in $A$ and maps $x$ to $R$; the image of $L \cup \set{x}$ under this substitution
  is $L \cup R$. By the closure of full trios under substitution with regular sets \cite[Theorem
  11.4]{hopcroft_automata}, $L \cup R \in \mathfrak{C}$.
\end{proof}

For the definitions of tree-adjoining languages, see \cite{kallmeyer_parsing}; for ET0L languages, see
\cite{rozenberg_extension}; and for indexed languages see \cite{aho_indexed}.

Recall that a \defterm{cross-section} (or \defterm{set of normal forms}) for a monoid $M$ is a language $L$ over some
generating set $A$ for $M$ such that every element of $M$ has a unique representative in $L$. Generally, we are
interested in cross-sections lying in some natural class of languages, which are usually made up of languages over a
finite alphabet (so that the cross section is over some finite generating set). A \defterm{Markov monoid} is a monoid
with a prefix-closed regular cross-section over some (necessarily finite) generating set; see \cite{cm_markov} for
background reading.

\begin{lemma}
  \label{lem:changegens}
  Let $M$ be a monoid that has a cross-section in a class of languages $\mathfrak{C}$ that is closed under
  homomorphisms. Let $A$ be a generating set for $M$. Then $M$ has a cross-section over $A$ in $\mathfrak{C}$.
\end{lemma}

[This technical result is trivial to prove, but does not seem to appear explicitly in the literature in the form
used in this paper, so a proof is supplied.]

\begin{proof}
  Let $M$ have a cross-section $K$ in $\mathfrak{C}$ with respect to some generating set $B$. For each $b \in B$, choose
  a word $u_b \in A^*$ representing $b$. Let $\phi : B^* \to A^*$ be the homomorphism that extends the map
  $b \mapsto u_b$. Let $L = \phi(K)$. Then $L$ is clearly also a cross-section for $M$, and, since $\mathfrak{C}$ is
  closed under homomorphisms, $L$ is also in $\mathfrak{C}$.
\end{proof}

\section{Direct factors}

This section is dedicated to using the construction described in \fullref{Subsection}{subsec:monoidactions} to define a
monoid that does not admit a regular cross-section, but whose direct product with $\zset$ admits a regular
cross-section.

Let $F$ be the free monoid with basis $\{x,y,y',z,z'\}$. Let
\begin{align*}
P &= \gset{p_{\alpha,\beta}}{\alpha \in \nset \cup \set{0}0, \beta \in \zset} \\
T &= P \cup \{\Omega\}
\end{align*}
For convenience, let $\bset = \gset{2^k}{k \in \nset \cup \set{0}}$. Define an action of the generators $x,y,y',z,z'$ on
$T$ as follows:
\begin{align*}
p_{\alpha,\beta} \cdot x &= \begin{cases}p_{\alpha+1,0} & \text{if $\beta = 0$,} \\
\Omega & \text{if $\beta \neq 0$;}
\end{cases} \\
p_{\alpha,\beta} \cdot y &=
\begin{cases}
p_{\alpha,\beta+1} & \text{if $\alpha \notin \bset$,} \\
p_{\alpha,0} & \text{if $\alpha \in \bset$ and $\beta = 0$,} \\
\Omega & \text{if $\alpha \in \bset$ and $\beta \neq 0$;}
\end{cases} \displaybreak[0]\\
p_{\alpha,\beta} \cdot y' &=
\begin{cases}
p_{\alpha,\beta-1} & \text{if $\alpha \notin \bset$,} \\
p_{\alpha,0} & \text{if $\alpha \in \bset$ and $\beta = 0$,} \\
\Omega & \text{if $\alpha \in \bset$ and $\beta \neq 0$;}
\end{cases} \displaybreak[0]\\
p_{\alpha,\beta} \cdot z &=
\begin{cases}
p_{\alpha,\beta} & \text{if $\alpha \notin \bset$;} \\
p_{\alpha,\beta+1} & \text{if $\alpha \in \bset$,}
\end{cases} \displaybreak[0]\\
p_{\alpha,\beta} \cdot z' &=
\begin{cases}
p_{\alpha,\beta} & \text{if $\alpha \notin \bset$.} \\
p_{\alpha,\beta-1} & \text{if $\alpha \in \bset$,}
\end{cases} \\
\Omega \cdot x &= \Omega \cdot y  = \Omega \cdot y'  = \Omega \cdot z  = \Omega \cdot z' = \Omega
\end{align*}
\fullref{Figure}{fig:action} illustrates the action of $F$ on $T$.

\begin{figure}[t]
  \centering
  \begin{tikzpicture}[x=16mm,y=12mm]
    \useasboundingbox (-1,-5) -- (8,5);
    \begin{scope}[every node/.style={rectangle with rounded corners,draw=gray,inner sep=.5mm,font=\small}]
      \foreach\x in {0,1,2,3,4} {
        \draw (\x,0) node (p\x 0) {$p_{\x,0}$};
        \foreach\y in {-1,1} {
          \draw (\x,\y) node[draw=white] (p\x\y) {};
        };
      };
      \foreach\x in {3,4} {
        \foreach\y in {-3,-2,-1,1,2,3} {
          \draw (\x,\y) node (p\x\y) {$p_{\x,\y}$};
        };
        \foreach\y in {-4,4} {
          \draw (\x,\y) node[draw=white] (p\x\y) {};
        };
      };
      \draw (7,0) node (omega) {$\Omega$};
    \end{scope}
    \begin{scope}[every node/.style={inner sep=.5mm,font=\footnotesize}]
      \foreach\x/\xnext in {0/1,1/2,2/3,3/4} {
        \draw[->] (p\x 0) edge node[auto] {$x$} (p\xnext 0);
      };
      \draw[->,densely dashed] (p40) edge node[auto] {$x$} ++(1,0);
      \foreach\y/\ynext in {-3/-2,-2/-1,-1/0,0/1,1/2,2/3} {
        \draw[->] (p3\y) edge[bend left] node[auto] {$y$} (p3\ynext);
        \draw[->] (p3\ynext) edge[bend left] node[auto] {$y'$} (p3\y);
      };
      \draw[->,densely dashed] (p33) edge[bend left] node[auto] {$y$} (p34);
      \draw[->,densely dashed] (p34) edge[bend left] node[auto] {$y'$} (p33);
      \draw[->,densely dashed] (p3-4) edge[bend left] node[auto] {$y$} (p3-3);
      \draw[->,densely dashed] (p3-3) edge[bend left] node[auto] {$y'$} (p3-4);
      \foreach\y/\ynext in {-3/-2,-2/-1,-1/0,0/1,1/2,2/3} {
        \draw[->] (p4\y) edge[bend left] node[auto] {$z$} (p4\ynext);
        \draw[->] (p4\ynext) edge[bend left] node[auto] {$z'$} (p4\y);
      };
      \draw[->,densely dashed] (p43) edge[bend left] node[auto] {$z$} (p44);
      \draw[->,densely dashed] (p44) edge[bend left] node[auto] {$z'$} (p43);
      \draw[->,densely dashed] (p4-4) edge[bend left] node[auto] {$z$} (p4-3);
      \draw[->,densely dashed] (p4-3) edge[bend left] node[auto] {$z'$} (p4-4);
      \foreach\x in {0} {
        \draw[->,densely dashed] (p\x 0) edge[bend left] node[auto] {$y$} (p\x 1);
        \draw[->,densely dashed] (p\x 1) edge[bend left] node[auto] {$y'$} (p\x 0);
        \draw[->,densely dashed] (p\x -1) edge[bend left] node[auto] {$y$} (p\x 0);
        \draw[->,densely dashed] (p\x 0) edge[bend left] node[auto] {$y'$} (p\x -1);
      };
      \foreach\x in {1,2} {
        \draw[->,densely dashed] (p\x 0) edge[bend left] node[auto] {$z$} (p\x 1);
        \draw[->,densely dashed] (p\x 1) edge[bend left] node[auto] {$z'$} (p\x 0);
        \draw[->,densely dashed] (p\x -1) edge[bend left] node[auto] {$z$} (p\x 0);
        \draw[->,densely dashed] (p\x 0) edge[bend left] node[auto] {$z'$} (p\x -1);
      }
      \foreach\y/\ynext in {1,2,3} {
        \draw[rounded corners=3mm] (p4\y) -- ++(.5,.5) |- (5,4);
        \draw[rounded corners=3mm] (p4-\y) -- ++(.5,-.5) |- (5,-4);
        \draw[rounded corners=3mm] (p3\y) -- ++(.5,.5) |- (4,4.5);
        \draw[rounded corners=3mm] (p3-\y) -- ++(.5,-.5) |- (4,-4.5);
      };
      \draw[<-,rounded corners=3mm] (omega) -- ++(-.5,.5) |- node[auto,pos=.75] {$x,y,y'$} (5,4);
      \draw[<-,rounded corners=3mm] (omega) -- ++(-.5,-.5) |- node[auto,pos=.75] {$x,y,y'$} (5,-4);
      \draw[->,rounded corners=3mm] (4,4.5) -| node[auto,pos=.25] {$x$} (omega);
      \draw[->,rounded corners=3mm] (4,-4.5) -| node[anchor=north,pos=.25] {$x$} (omega);
    \end{scope}
  \end{tikzpicture}
  \caption{Diagram of the action of $A$ on $T$. Actions which fix
    points of $T$ (which would be loops at some vertex) are omitted
    for clarity. The elements $p_{3,\beta}$ and elements $p_{4,\beta}$
    are shown to illustrate the different action on $p_{\alpha,\beta}$
    when $\alpha$ is and is not a power of $2$.}
  \label{fig:action}
\end{figure}
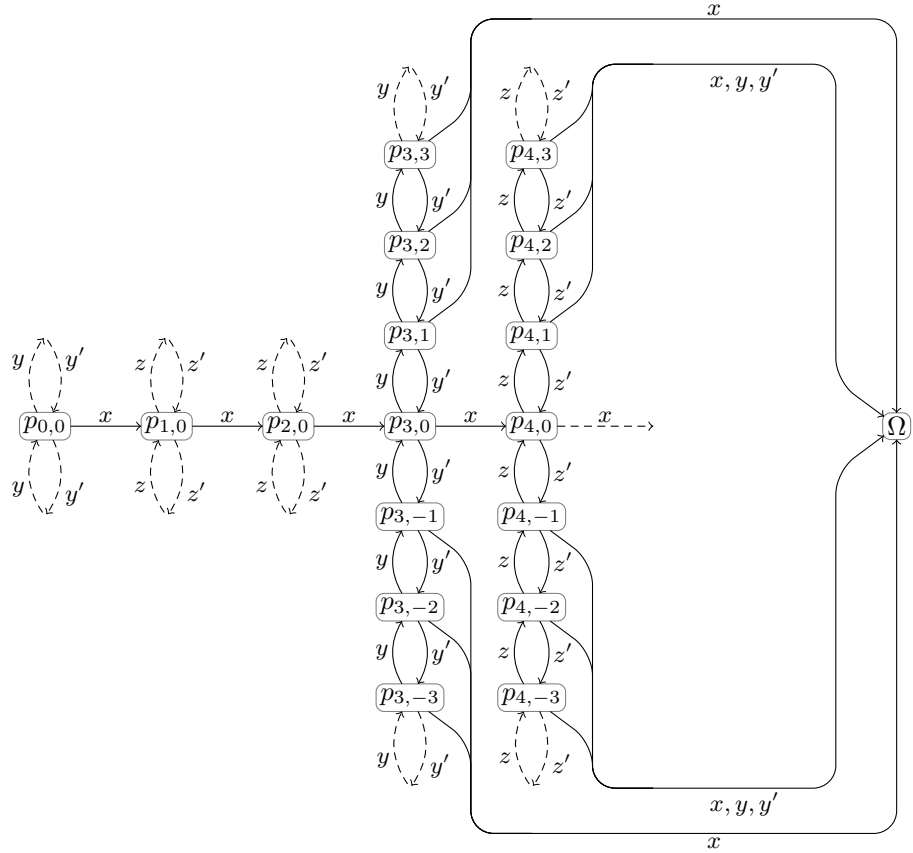

\begin{proposition}
  \label{prop:dirprodcrosssection}
  The direct product $F[T] \times \zset$ has a prefix-closed regular cross-section.
\end{proposition}

\begin{proof}
  Let
  \[
    A = \set{a,b_0,b_1,b'_0,b'_{-1},c_0,c_1,c'_0,c'_{-1},d_1,d_{-1},e,f},
  \]
  where
  \begin{align*}
    a &= (x,0) \\
    b_0 &= (y,0) & b'_0 &= (y',0) \\
    b_1 &= (y,1) & b'_{-1} &= (y',-1) \\
    c_0 &= (z,0) & c'_0 &= (z',0) \\
    c_1 &= (z,1) & c'_{-1} &= (z',-1) \\
    d_1 &= (\emptyword,1) & d_{-1} &= (\emptyword,-1) \\
    e &= (p_{0,0},0) \\
    f &= (\Omega,0).
  \end{align*}
  Let
  \begin{align*}
    L_1 &= ea^*b_1^*c_1^* \cup ea^*b_1^*{c'_{-1}}^* \cup ea^*{b'_{-1}}^*c_1^* \cup ea^*{b'_{-1}}^*{c'_{-1}}^* \cup fd_1^* \cup fd_{-1}^*, \\
    L_2 &= \set{a,b_0,b'_0,c_0,c'_0}^*d_1^* \cup \set{a,b_0,b'_0,c_0,c'_0}^*d_{-1}^*,
  \end{align*}
  and let $L = L_1 \cup L_2$. By inspection, the regular language $L$ is prefix-closed.

  We claim that $L$ is a regular cross-section for $F[T] \times \zset$. Note first that every word in $L_1$ is non-empty
  and begins with $e$ or $f$, whereas no word in $L_2$ begins with $e$ or $f$. Hence the union $L_1 \cup L_2$ is disjoint. The aim is
  now to show that every element of $T \times \zset$ has a unique representative in $L_1$ and that every element of
  $F \times \zset$ has a unique representative in $L_2$.

  It is easy to see that $L_2$ maps bijectively onto $F \times \zset$, for the prefix in $\set{a,b_0,b'_0,c_0,c'_0}^*$
  maps bijectively onto $F \times \set{0}$ while the suffix in $d_1^* \cup d_{-1}^*$ is uniquely determined by the
  $\zset$-coordinate.

  To show that $L_1$ maps bijectively onto $T \times \zset$, note first that words of the form $fd_1^* \cup fd_{-1}^*$
  map bijectively onto $\{\Omega\} \times \zset$.

  It remains to show that the set of words of the form
  \[
    ea^\alpha b_1^\beta c_1^\gamma,\quad ea^\alpha b_1^\beta {c'_{-1}}^\gamma,\quad ea^\alpha {b'_{-1}}^\beta c_1^\gamma,\quad ea^\alpha {b'_{-1}}^\beta {c'_{-1}}^\gamma
  \]
  maps bijectively onto $P \times \zset$. In each case, the prefix $ea^\alpha$ labels a path in the Cayley graph of
  $F[T]$ that leads to the element $(p_{\alpha,0},0)$. The paths then stay within the set
  $\gset{p_{\alpha,\beta}}{\beta \in \zset} \times \zset$, but `look' quite different depending on whether
  $\alpha \in \bset$ (see \fullref{Figure}{fig:pathsinbset}) or $\alpha \notin \bset$ (see
  \fullref{Figure}{fig:pathsnotinbset}).
  \begin{figure}[t]
    \centering
    \begin{tikzpicture}[x=17mm,y=12mm]
      \useasboundingbox (-3,-3) -- (3,3);
      \begin{scope}[every node/.style={rectangle with rounded corners,draw=gray,inner sep=.5mm,font=\scriptsize}]
        \foreach\y in {-3,-2,-1,0,1,2,3} {
          \foreach\x in {-3,-2,-1,0,1,2,3} {
            \draw (\x,\y) node (p\x\y) {$(p_{\alpha,\x},\y)$};
          };
          \foreach\x in {-4,4} {
            \draw (\x,\y) node[draw=white] (p\x\y) {};
          };
        };
      \end{scope}
      \begin{scope}[every node/.style={inner sep=.2mm,font=\scriptsize}]
        \foreach\y/\ynext in {-3/-2,-2/-1,-1/0,0/1,1/2,2/3} {
          \ifnum \y < 0 %
          \draw[->] (p0\ynext) edge node[auto] {$b'_{-1}$} (p0\y);
          \else
          \draw[->] (p0\y) edge node[auto] {$b_1$} (p0\ynext);
          \fi
        };
        \foreach\x/\xnext in {-3/-2,-2/-1,-1/0,0/1,1/2,2/3} {
          \ifnum \x < 0 %
          \draw[->] (p\xnext\xnext) edge node[auto] {$c'_{-1}$} (p\x\x);
          \else
          \draw[->] (p\x\x) edge node[auto] {$c_1$} (p\xnext\xnext);
          \fi
        };
        \draw[->] (p0-1) edge node[auto] {$c_1$} (p10);
        \draw[->] (p10) edge node[auto] {$c_1$} (p21);
        \draw[->] (p21) edge node[auto] {$c_1$} (p32);
        \draw[->] (p0-2) edge node[auto] {$c_1$} (p1-1);
        \draw[->] (p1-1) edge node[auto] {$c_1$} (p20);
        \draw[->] (p20) edge node[auto] {$c_1$} (p31);
        \draw[->] (p0-3) edge node[auto] {$c_1$} (p1-2);
        \draw[->] (p1-2) edge node[auto] {$c_1$} (p2-1);
        \draw[->] (p2-1) edge node[auto] {$c_1$} (p30);
        \draw[->] (p1-3) edge node[auto] {$c_1$} (p2-2);
        \draw[->] (p2-3) edge node[auto] {$c_1$} (p3-2);
        \draw[->] (p2-2) edge node[auto] {$c_1$} (p3-1);
        \draw[->] (p01) edge node[auto] {$c_1$} (p12);
        \draw[->] (p12) edge node[auto] {$c_1$} (p23);
        \draw[->] (p02) edge node[auto] {$c_1$} (p13);
        \draw[->] (p0-2) edge node[auto] {$c'_{-1}$} (p-1-3);
        \draw[->] (p0-1) edge node[auto] {$c'_{-1}$} (p-1-2);
        \draw[->] (p-1-2) edge node[auto] {$c'_{-1}$} (p-2-3);
        \draw[->] (p01) edge node[auto] {$c'_{-1}$} (p-10);
        \draw[->] (p-10) edge node[auto] {$c'_{-1}$} (p-2-1);
        \draw[->] (p-2-1) edge node[auto] {$c'_{-1}$} (p-3-2);
        \draw[->] (p02) edge node[auto] {$c'_{-1}$} (p-11);
        \draw[->] (p-11) edge node[auto] {$c'_{-1}$} (p-20);
        \draw[->] (p-20) edge node[auto] {$c'_{-1}$} (p-3-1);
        \draw[->] (p03) edge node[auto] {$c'_{-1}$} (p-12);
        \draw[->] (p-12) edge node[auto] {$c'_{-1}$} (p-21);
        \draw[->] (p-21) edge node[auto] {$c'_{-1}$} (p-30);
        \draw[->] (p-13) edge node[auto] {$c'_{-1}$} (p-22);
        \draw[->] (p-22) edge node[auto] {$c'_{-1}$} (p-31);
        \draw[->] (p-23) edge node[auto] {$c'_{-1}$} (p-32);
      \end{scope}
    \end{tikzpicture}
    \caption{Diagram of paths labelled by elements of $L_1$ in the Cayley graph of $F[T] \times \zset$, showing the
      paths passing through the subset $\gset{p_{\alpha,\beta}}{\beta \in \zset} \times \zset$, where
      $\alpha \in \bset$.}
    \label{fig:pathsinbset}
  \end{figure}
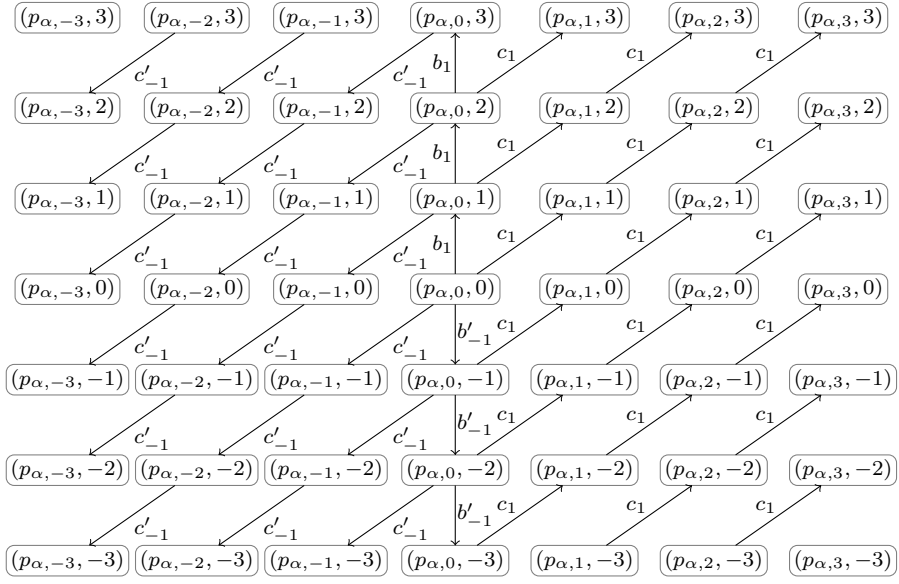
  \begin{figure}[t]
    \centering
    \begin{tikzpicture}[x=17mm,y=12mm]
      \useasboundingbox (-3,-3) -- (3,3);
      \begin{scope}[every node/.style={rectangle with rounded corners,draw=gray,inner sep=.5mm,font=\scriptsize}]
        \foreach\y in {-3,-2,-1,0,1,2,3} {
          \foreach\x in {-3,-2,-1,0,1,2,3} {
            \draw (\x,\y) node (p\x\y) {$(p_{\alpha,\x},\y)$};
          };
          \foreach\x in {-4,4} {
            \draw (\x,\y) node[draw=white] (p\x\y) {};
          };
        };
      \end{scope}
      \begin{scope}[every node/.style={inner sep=.2mm,font=\scriptsize}]
        \foreach\x/\xnext in {0/1,1/2,2/3} {
          \draw[->] (p\x\x) edge node[auto] {$b_1$} (p\xnext\xnext);
        };
        \foreach\x/\xnext in {0/-1,-1/-2,-2/-3} {
          \draw[->] (p\x\x) edge node[auto] {$b'_{-1}$} (p\xnext\xnext);
        };
        \foreach \x in {-3,-2,-1,0,1,2,3} {
          \foreach\y/\ynext in {-3/-2,-2/-1,-1/0,0/1,1/2,2/3} {
            \ifnum \y < \x %
            \draw[->] (p\x\ynext) edge node[auto] {$c'_{-1}$} (p\x\y);
            \else
            \draw[->] (p\x\y) edge node[auto] {$c_1$} (p\x\ynext);
            \fi
          };
        };
      \end{scope}
    \end{tikzpicture}
    \caption{Diagram of paths labelled by elements of $L_1$ in the Cayley graph of $F[T] \times \zset$, showing the
      paths passing through the subset $\gset{p_{\alpha,\beta}}{\beta \in \zset} \times \zset$, where
      $\alpha \notin \bset$.}
    \label{fig:pathsnotinbset}
  \end{figure}
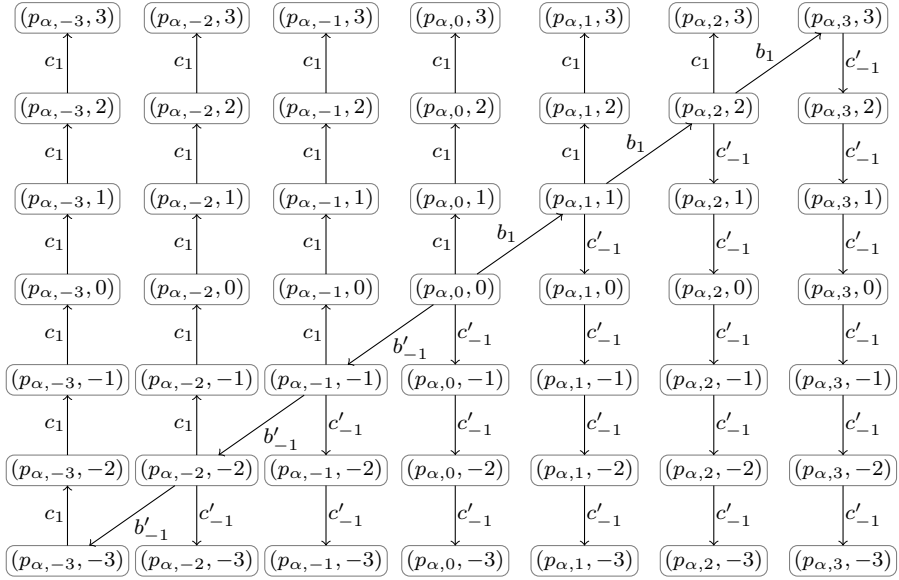
  To prove bijectivity formally, proceed as follows. First, note that
  \begin{align*}
    ea^\alpha b_1^\beta c_1^\gamma &= (p_{0,0},0)(x,0)^\alpha (y,1)^\beta (z,1)^\gamma \\
                                   &= (p_{\alpha,0},0)(y,1)^\beta (z,1)^\gamma \displaybreak[0]\\
                                   &= \begin{cases}
                                     (p_{\alpha,\beta},\beta) (z,1)^\gamma & \text{if $\alpha \neq 2^k$ for any $k$} \\
                                     (p_{\alpha,0},\beta) (z,1)^\gamma & \text{if $\alpha = 2^k$ for some $k$}
                                         \end{cases} \displaybreak[0]\\
                                   &= \begin{cases}
                                     (p_{\alpha,\beta},\beta+\gamma) & \text{if $\alpha \neq 2^k$ for any $k$} \\
                                     (p_{\alpha,\gamma},\beta+\gamma) & \text{if $\alpha = 2^k$ for some $k$.}
                                   \end{cases}
  \end{align*}
  Similarly,
  \begin{align*}
    ea^\alpha b_1^\beta {c'_{-1}}^\gamma &= \begin{cases}
      (p_{\alpha,\beta},\beta-\gamma) & \text{if $\alpha \neq 2^k$ for any $k$} \\
      (p_{\alpha,-\gamma},\beta-\gamma) & \text{if $\alpha = 2^k$ for some $k$;}
    \end{cases} \displaybreak[0]\\
    ea^\alpha {b'_{-1}}^\beta c_1^\gamma &= \begin{cases}
      (p_{\alpha,-\beta},-\beta+\gamma) & \text{if $\alpha \neq 2^k$ for any $k$} \\
      (p_{\alpha,\gamma},-\beta+\gamma) & \text{if $\alpha = 2^k$ for some $k$;}
    \end{cases} \displaybreak[0]\\
    ea^\alpha {b'_{-1}}^\beta {c'_{-1}}^\gamma &= \begin{cases}
      (p_{\alpha,-\beta},-\beta-\gamma) & \text{if $\alpha \neq 2^k$ for any $k$} \\
      (p_{\alpha,-\gamma},-\beta-\gamma) & \text{if $\alpha = 2^k$ for some $k$.}
    \end{cases}
  \end{align*}

  Let $(p_{\alpha,\delta},\zeta) \in T \times \zset$. There are eight
  cases to consider, depending on whether $\alpha = 2^k$ for some $k$;
  whether $\delta \geq 0$ or $\delta < 0$; and whether $\zeta \geq
  \delta$ or $\zeta < 0$. We do one exemplary case: $\alpha \neq 2^k$
  for any $k$, $\delta \geq 0$, and $\zeta < \delta$. Then
  \[
    ea^{\alpha}b_1^\delta {c'_{-1}}^{\delta-\zeta} = (p_{\alpha,\delta},\delta-(\delta-\zeta)) = (p_{\alpha,\delta},\zeta).
  \]
  Furthermore, checking the other cases shows that this is the unique
  element of $L_1$ representing $(p_{\alpha,\delta},\zeta)$.
\end{proof}

\begin{proposition}
  \label{prop:dirfactornocrosssection}
  The monoid $F[T]$ does not admit a regular cross-section.
\end{proposition}

\begin{proof}
  Suppose for \textit{reductio ad absurdum} that $F[T]$ admits a regular cross-section. Then by
  \fullref{Lemma}{lem:changegens}, $F[T]$ admits a regular cross-section $L$ over
  $A = \set{x,y,y',z,z',p_{0,0},\Omega}$, since $A$ generates $F[T]$ (because $A$ includes a generating set for $F$ and
  and the image of $p_{0,0}$ under the action of $F$ is the whole of $T$). Since regularity is preserved on replacing a
  single word, assume that $\Omega \in L$.

  For any word $w$ over $A$, one of the following three cases holds:
  \begin{enumerate}
  \item $p_{0,0}$ and $\Omega$ do not appear in $w$, in which case $w$ represents an element of $F$;
  \item there is an occurrence of $\Omega$ in $w$ such that there is no later occurrence of $p_{0,0}$ in $w$, in which
    case $w$ represents $\Omega$;
  \item there is an occurrence of $p_{0,0}$ in $w$ such that there is no later occurrence of $\Omega$ in $w$, in which
    case $w$ represents the same element of $F[T]$ as the suffix of $w$ starting with the last occurrence of $p_{0,0}$.
  \end{enumerate}

  Let $L^{(1)}$ be the language consisting of suffixes of words in $L$ that lie in $p_{0,0}\set{x,y,y',z,z'}^*$. Note
  that $L^{(1)}$ can be obtained by applying a GSM mapping to $L$, and so $L^{(1)}$ is regular. (The required GSM
  initially reads input symbols without producing output. It non-deterministically chooses some point at which to
  produce output, and subsequently outputs each symbol it reads, checking that the suffix after the chosen point is of
  the required form.) Note also that $\Omega \in L \setminus L^{(1)}$. All words in $L$ that have suffixes in $L^{(1)}$
  satisfy case~3 above. Thus the language $L^{(1)}$ must map bijectively to $T \setminus \set{\Omega}$.

  Let $p_{0,0}v \in L^{(1)}$. Suppose that $v = v'xwxv''$, where $w \in \set{y,y',z,z'}^*$. Then $v'x$ must send
  $p_{0,0}$ to some $p_{\alpha,0}$ (for the symbol $x$ would send any other point to $\Omega$, contradicting the fact
  that $\Omega \in L$ is a unique representative). Further, if $p_{\alpha,0} \cdot w = p_{\alpha,\beta}$ for
  $\beta \neq 0$ (note that the action of $w$ cannot alter $\alpha$), then
  $p_{0,0}\cdot v = p_{\alpha,\beta} \cdot xv'' = \Omega$, which is again a contradiction. Hence
  $p_{\alpha,0} \cdot w = p_{\alpha,0}$. Thus deleting the subword $w$ from $v$ does not alter the represented
  element. Similar reasoning applies if $v = wxv''$ for some $w \in \set{y,y',z,z'}^*$; again $w$ can be deleted without
  altering the represented element. Apply a GSM mapping to $L^{(1)}$ that deletes any string of symbols from
  $\set{y,y',z,z'}$ between $p_{0,0}$ and $x$ or between two symbols $x$; this yields a regular language
  $L^{(2)} \subseteq p_{0,0}x^*\set{y,y',z,z'}$ mapping bijectively to $T \setminus \set{\Omega}$. (The required GSM
  initially outputs each symbol it reads. On encountering a symbol $p_{0,0}$ or $x$, it non-deterministically guesses
  whether there now follows a string of symbols from $\set{y,y',z,z'}$ followed by another $x$. If it guesses `yes', it
  checks whether what follows is a string of this form, producing no output until it reaches the $x$, when it resumes
  outputting each symbol it reads. If it guesses `no', it continues to output each symbol it reads, checking that what
  follows is not a string of the given form.)

  Now let $p_{0,0}v \in L^{(2)}$. Suppose that $v = v'ywyv''$ for $w \in \set{z,z'}$ and that
  $p_{0,0}\cdot v'y = p_{\alpha,\beta}$. If $\alpha \notin \bset$, then $p_{\alpha,\beta}\cdot w = p_{\alpha,\beta}$ and
  so deleting $w$ does not alter the represented element. So assume $\alpha \in \bset$. Then
  $p_{0,0}\cdot v' = p_{0,0}\cdot v'y = p_{\alpha,0}$, since otherwise the action of $y$ would lead to
  $\Omega$. Similarly, $p_{0,0}\cdot v'yw = p_{0,0}\cdot v'ywy = p_{\alpha,0}$. So $w$ fixes $p_{0,0}\cdot v'y$ and so
  deleting $w$ does not alter the represented element.

  The same reasoning applies if one replaces either or both of the distinguished symbols $y$ by $y'$. Apply a GSM
  mapping to $L^{(2)}$ that deletes any string of symbols from $\set{z,z'}$ between two symbols from $\set{y,y'}$; this
  yields a regular language $L^{(3)} \subseteq p_{0,0}x^*\set{z,z'}^*\set{y,y'}^*\set{z,z'}^*$ that maps bijectively to
  $T \setminus \set{\Omega}$. (The required GSM functions similarly to the last-described one.)

  Let $p_{0,0}v \in L^{(3)}$. Suppose that $v = x^\alpha wyv''$ for $w \in \set{z,z'}^*$. Suppose
  $p_{0,0}\cdot x^\alpha w = p_{\alpha,\beta}$. Then $\beta = 0$, for otherwise the action of $y$ would lead to
  $\Omega$. Hence $w$ fixes $p_{0,0}\cdot x^\alpha$ and so deleting $w$ does not alter the represented element. The same
  reasoning applies if one replaces the distinguished symbol $y$ by $y'$.  Apply a GSM mapping to $L^{(3)}$
  that deletes any string of symbols from $\set{z,z'}$ between $p_{0,0}$ and a symbol from $\set{y,y'}$; this yields a
  regular language $L^{(4)} \subseteq p_{0,0}x^*\set{y,y'}^*\set{z,z'}^*$ that maps bijectively to
  $T \setminus \set{\Omega}$. (Again, the required GSM functions similarly to the last-described one.)

  Let $n$ be larger than the number of states in an automaton recognizing $L^{(4)}$. Consider $p_{2^n,n} \in T$. Let
  $p_{0,0}x^{2^n}vw \in L^{(4)}$, where $v \in \set{y,y'}^*$ and $w \in \set{z,z'}^*$, be the unique word representing
  $p_{2^n,n}$. Then since at least $n$ symbols $z$ are required to reach $p_{2^n,n}$ from $p_{2^n,0}$, the word $w$ has
  length at least $n$.

  By the pumping lemma, $p_{0,0}x^{2^n-k}vw \in L^{(4)}$ for some $k$ such that $0 < k < n$. Let
  $p_{\alpha,\beta} = p_{0,0}x^{2^n-k}v$; then $\alpha = 2^n - k \notin \bset$, so that $z$ and $z'$ fix
  $p_{\alpha,\beta}$. Since $w$ has length at least $n$, it factors as $w = pqr$ such that
  $p_{0,0}x^{2^n-k}vpq^ir \in L^{(4)}$ for all $i \in \nset \cup \set{0}$. But $z$ and $z'$ fix $p_{\alpha,\beta}$, so
  all these words represent $p_{\alpha,\beta}$, which contradicts the fact that $L^{(4)}$ maps bijectively onto
  $T \setminus \set{\Omega}$.
\end{proof}

Combining \fullref{Propositions}{prop:dirprodcrosssection} and \ref{prop:dirfactornocrosssection} yields the following result:

\begin{theorem}
  \begin{enumerate}
  \item The class of monoids with regular cross-sections is not closed under taking direct factors.
  \item The class of Markov monoids is not closed under taking direct factors.
  \end{enumerate}
\end{theorem}

However, $F[T]$ has a cross-section in the next natural language class that strictly contains the regular languages:

\begin{proposition}
  $F[T]$ has a one-counter cross-section.
\end{proposition}

\begin{proof}
  Let $A = \set{x,y,y',z,z',p_{0,0},\Omega}$ and let $L$ be the one-counter language
  \[
    \set{x,y,y',z,z'}^* \cup p_{0,0}x^*\gset{y^nz^n,(y')^n(z')^n}{n \in \nset}.
  \]
  Clearly, $\set{x,y,y',z,z'}^*$ maps bijectively onto $F$. Further, $p_{0,0}\cdot x^\alpha y^nz^n = p_{\alpha,n}$ and
  $p_{0,0}\cdot x^\alpha (y')^n(z')^n = p_{\alpha,-n}$ (regardless of whether $\alpha \in \bset$), and so
  $p_{0,0}x^*\gset{y^nz^n,(y')^n(z')^n}{n \in \nset}$ maps bijectively onto $T$.
\end{proof}

\section{Cross-sections in different classes of languages}

This section is dedicated to showing that we can `separate' two full trios, one not contained in the other, using monoid
cross-sections, and applying this result to some particular interesting language classes.

\begin{theorem}
  \label{thm:cd}
  Let $\mathfrak{C}$ and $\mathfrak{D}$ be full trios such that $\mathfrak{C}$ is not a subclass of $\mathfrak{D}$.
  Then there is a monoid with a cross-section in $\mathfrak{C}$ but no cross-section in $\mathfrak{D}$.
\end{theorem}

\begin{proof}
  Since $\mathfrak{C}$ is not a subclass of $\mathfrak{D}$, there is a language $K$ over some finite alphabet $B$ that
  is in $\mathfrak{C}$ but not in $\mathfrak{D}$. Since full trios always contain the empty language, $K$ is thus
  non-empty. Let
  \[
    T = \gset{p_u}{u \in B^*} \cup \gset{q_u}{u \in K} \cup \set{\Omega}.
  \]
  Let $z$ be a new symbol not in $B$, let $F$ be the free monoid on $B \cup \set{z}$, and let $F$ act on $T$ as follows
  \begin{align*}
    p_u \cdot b &= p_{ub} &&\text{for $u \in  B^*$ and $b \in B$;} \\
    p_u \cdot z &= \begin{cases}
      q_u & \text{if $u \in K$,} \\
      \Omega & \text{if $u \notin K$,}
    \end{cases} &&\text{for $u \in B^*$;} \\
    q_u \cdot x &= \Omega && \text{for $u \in K$ and $x \in B \cup \set{z}$;} \\
    \Omega \cdot x &= \Omega && \text{for $x \in B \cup \set{z}$.}
  \end{align*}
  The aim is now to show that the monoid $F[T]$ has a cross-section in $\mathfrak{C}$ but not in $\mathfrak{D}$.

  Let $A = B \cup \set{z,p_\emptyword,\Omega}$; note that $A$ is a generating set for $F[T]$. Let
  \[
    L = \parens{B\cup\set{z}}^* \cup p_\emptyword B^* \cup p_\emptyword K z.
  \]
  Deletion of a fixed prefix $p_\emptyword$ and a fixed suffix $z$ can be performed by a GSM mapping. Since full trios
  are closed under inverse GSM mappings, and since $K$ is in $\mathfrak{C}$, it follows that $p_\emptyword K z$ is also
  in $\mathfrak{C}$. Hence, by the closure of full trios under the operation of union with a regular language
  (\fullref{Lemma}{lem:fulltrioclosureunionreg}), $L$ is in $\mathfrak{C}$. Finally, it is clear that $L$ is a
  cross-section for $F[T]$ since the sublanguage $\parens{B\cup\set{z}}^*$ maps bijectively onto $F$, the sublanguage
  $p_\emptyword B^*$ maps bijectively onto $\gset{p_u}{u \in B^*}$, and the sublanguage $p_\emptyword K z$ maps
  bijectively onto $\gset{q_u}{u \in K}$, and these three sublanguages, whose union is $L$, are disjoint.

  Now suppose for \textit{reductio ad absurdum} that $F[T]$ admits a cross-section in $\mathfrak{D}$. By
  \fullref{Lemma}{lem:changegens}, assume without loss of generality that the cross-section is a language $L$ over
  $B \cup \set{z,p_\emptyword,\Omega}$. Since replacing a single word in $L$ can be carried out via intersection with a
  regular language and then union with a (one-element) regular language (see
  \fullref{Lemma}{lem:fulltrioclosureunionreg}), assume that $\Omega \in L$.

  Let $L'$ be the language consisting of suffixes of words in $L$ that lie in $p_\emptyword v$, where
  $v \in (B \cup \set{z})^*$. Note that $L'$ can be obtained by applying a GSM mapping to $L$, and so
  $L' \in \mathfrak{D}$. (The GSM that extracts the required suffixes functions similarly to the one in the proof of
  \fullref{Proposition}{prop:dirfactornocrosssection}.) By definition of multiplication in $F[T]$, and noting that
  $\Omega \in L \setminus L'$, the words in $L'$ must map bijectively to $T \setminus \set{\Omega}$. By definition of
  the action of $F$ on $T$, the words in $L'$ that map bijectively to $\gset{q_u}{u \in K}$ are precisely those of the
  form $p_\emptyword u z$, where $u \in K$. Intersection with a regular language gives the language of the words $u$
  such that $p_\emptyword uz \in L'$, showing that $K \in \mathfrak{D}$. But this contradicts the fact that
  $K \notin \mathfrak{D}$. Hence $F[T]$ does not have a cross-section in $\mathfrak{D}$.
\end{proof}

Each of the following language classes forms a full trio: the indexed languages \cite[Theorem~3.2]{aho_indexed}; ET0L
languages \cite[Theorem 18]{rozenberg_extension}; the tree-adjoining languages (TAL) \cite[p.~59]{kallmeyer_parsing};
context-free languages (CFL) \cite[p.~278]{hopcroft_automata}.

Every ET0L language is indexed \cite[Corollary~4.1]{culik_onsomefamilies}, and the class of TALs is equivalent to
the languages recognised by a restricted type of indexed grammars called linear indexed grammars
\cite[p.~72]{kallmeyer_parsing}, hence the TALs are also indexed. The class of context-free languages is contained in
both the ET0L languages \cite{rozenberg_extension} and in the TALs \cite[p.~58]{kallmeyer_parsing}. (These containments
are illustrated in \fullref{Figure}{fig:containments}.)

\begin{figure}[t]
  \begin{tikzpicture}[scale=1.5]
    \node (indexed) at (0,1) {Indexed};
    \node (tal) at (-1,0) {ET0L};
    \node (et0l) at (1,0) {TAL};
    \node (talcapet0l) at (0,-1) {$\text{ET0L} \cap \text{TAL}$};
    \node (cfl) at (0,-2) {CFL};
    \begin{scope}[every node/.style={font=\scriptsize}]
      \draw (indexed) edge node[left] {$L_1$} (tal);
      \draw (indexed) edge node[right] {$L_2$} (et0l);
      \draw (tal) edge node[left] {$L_2$} (talcapet0l);
      \draw (et0l) edge node[right] {$L_1$} (talcapet0l);
      \draw (talcapet0l) edge node[auto] {$L_3$} (cfl);
    \end{scope}
  \end{tikzpicture}
  \caption{Containment of selected classes of languages.}
  \label{fig:containments}
\end{figure}
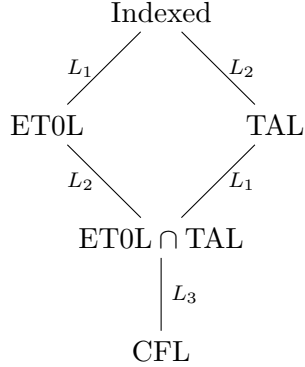

The aim is now to show that, for each containment $\mathfrak{D}\subsetneq \mathfrak{C}$ in
\fullref{Figure}{fig:containments}, there is a monoid with a cross-section in $\mathfrak{C}$ but that does not admit a
cross-section in $\mathfrak{D}$. By \fullref{Theorem}{thm:cd}, it suffices to exhibit a language that is in
$\mathfrak{C}$ but not in $\mathfrak{D}$.

\begin{corollary}
  There exists a monoid that admits a TAL (and hence indexed) cross-section but does not admit an ET0L cross-section.
\end{corollary}

\begin{proof}
  Let $L\subseteq X^*$ be any context-free but non-EDT0L language (the EDT0L languages are a proper subclass of the ET0L
  languages (note the `D' in the name of the subclass); for the existence of context-free non-EDT0L languages, see
  \cite{ehrenfeucht_onsome}).  Let $X'$ be a copy of $X$ and $\phi:X^*\rightarrow (X')^*$ the homomorphism defined by
  $x\phi = x'$.  Then the language $L_1 = \gset{w (w\phi)^\rev }{ w\in L }$ is generated by a linear indexed grammar
  \cite[Corollary~3.6]{duske_linear} (note that the notion of linear indexed grammar in \cite{duske_linear} is more
  restricted than the one in \cite{kallmeyer_parsing} equivalent to tree-adjoining grammars) and is hence a TAL and
  therefore also indexed.  However, $L_1$ is not ET0L \cite{ehrenfeucht_relationship}. The result follows by
  \fullref{Theorem}{thm:cd}.
\end{proof}

\begin{corollary}
  There exists a monoid that admits an ET0L (and hence indexed) cross-section but admits no TAL cross-section.
\end{corollary}

\begin{proof}
  Let $L_2 = \gset{ www }{ w\in \set{a,b}^*}$.  Then $L_2$ is generated by an ET0L-system with $4$ tables, each
  containing a single non-trivial production. The non-trivial productions are $S\rightarrow TTT$, $T\rightarrow aT$,
  $T\rightarrow bT$, and $T\rightarrow \emptyword$.  However, $L_2$ is not a TAL \cite[Lemma
  4.15]{kallmeyer_parsing}. The result follows by \fullref{Theorem}{thm:cd}.
\end{proof}

\begin{corollary}
  There exists a monoid that admits a cross-section that is both ET0L and TAL but admits no context-free cross-section.
\end{corollary}

\begin{proof}
  Finally, let $L_3 = \gset{ a^n b^n c^n }{ n\in \nset}$.  Then $L_3$ is TAL \cite[Problem~4.1]{kallmeyer_parsing} and
  ET0L, since it is generated by the following $3$-table ET0L-system:
  \begin{align*}
    S &\rightarrow ABC; \\
    A &\rightarrow a,\quad B\rightarrow b,\quad C\rightarrow c;\\
    A&\rightarrow \emptyword, \quad B\rightarrow \emptyword, \quad C\rightarrow \emptyword.
  \end{align*}
  However, $L_3$ is easily shown not to be context-free by a pumping argument. The result follows by
  \fullref{Theorem}{thm:cd}.
\end{proof}

\bibliography{\jobname}
\bibliographystyle{alphaabbrv}

\end{document}